\documentclass{birkjour}  %

 \newtheorem{thm}{Theorem}[section]
 \newtheorem{cor}[thm]{Corollary}
 \newtheorem{lem}[thm]{Lemma}
 
 \theoremstyle{definition}
 \newtheorem{defn}[thm]{Definition}
 \theoremstyle{remark}
 
 \newtheorem*{ex}{Example}
 \numberwithin{equation}{section}

 \usepackage{graphicx}
\usepackage{amsmath}
\usepackage{amssymb}
\usepackage{enumerate}
\usepackage{rawfonts}
\usepackage{setspace}

\usepackage{arydshln}
\usepackage{lscape}

\usepackage{float}
\usepackage{color}
\usepackage{multicol}

\usepackage[all]{xy}
\xyoption{poly} \xyoption{graph} \xyoption{arc} \xyoption{web}
\xyoption{2cell}

\begin{document}

%
%
%
%
%
%
%
%
%

\title{$3$-connected graphs and their degree sequences}

\author{Jonathan McLaughlin}

\address{%
Department of Mathematics, \\ 
St. Patrick's College,\\
 Dublin City University,\\ 
 Dublin 9, \\ 
 Ireland }

\email{jonny$\_\;$mclaughlin@hotmail.com}




\date{\today}

\begin{abstract} Necessary and sufficient conditions for a sequence of positive integers to be the degree sequence of a $3$-connected simple graph are detailed. Conditions are also given under which such a sequence is necessarily $3$-connected i.e. the sequence can only be realised as a $3$-connected graph. Finally, a matrix is introduced whose non-empty entries partition the set of $3$-connected graphs.  
\end{abstract}

\maketitle

\section{Introduction}
\parindent=0cm

Necessary and sufficient conditions for a sequence of non-negative integers to be connected i.e. the degree sequence of some finite simple connected graph, are implicit in Hakimi \cite{Hk} and have been stated explicitly by the author in \cite{Me15}. This note builds upon these conditions of Hakimi and begins by describing an iterative construction of a $3$-connected graph (in which all intermediate graphs are also $3$-connected), due to Barnette and Gr\"unbaum \cite{BG2}. The set of all $3$-connected graphs is then partitioned into equivalence classes each of which is placed separately in a unique entry of a matrix $P_{\mathcal{G}_3}$. The relationship between entries in $P_{\mathcal{G}_3}$ and the Barnette and Gr\"unbaum construction is then described. Necessary and sufficient conditions are given for a sequence of non-negative integers to be $3$-connected and/or necessarily $3$-connected.  Finally, the relationship between entries in $P_{\mathcal{G}_3}$ and (necessarily) $3$-connected degree sequences is outlined.

\section{Preliminaries }

Let $G=(V_{G},E_{G})$ be a graph where $V_{G}$ denotes the vertex set of $G$ and $E_{G}\subseteq [V_{G}]^{2}$ denotes the edge set of $G$ (given that $[V_G]^2$ is the set of all $2$-element subsets of $V_G$).  An edge $\{a,b\}$ is denoted $ab$. A graph is finite when $|V_{G}|<\infty$ and $|E_{G}|<\infty$, where $|X|$ denotes the cardinality of the set $X$. The union of graphs $G$ and $H$ i.e. $(V_{G}\cup V_{H}, E_{G}\cup E_{H})$, is denoted $G\cup H$. By a slight abuse of notation, $ab\cup G$ is understood to be the graph $(\{a,b\},\{ab\})\cup (V_G, E_G)$. A graph is simple if it contains no loops (i.e. $aa\not\in E_{G}$) or parallel/multiple edges (i.e. $\{ab,ab\}\not\subseteq E_{G}$). The {\it degree} of a vertex $v$ in a graph $G$, denoted $deg(v)$, is the number of edges in $G$ which contain $v$. A graph where all vertices have degree $k$ is called a {\it $k$-regular} graph. A {\it path} is a graph with $n$ vertices in which two vertices, known as the {\it endpoints}, have degree $1$ and $n-2$ vertices have degree $2$. A graph is {\it connected} if there exists at least one path between every pair of vertices in the graph. Paths $P_1$ and $P_2$, both with endpoints $a$ and $b$, are {\it internally disjoint} if $P_1\cap P_2=(\{a,b\},\{\})$. A graph $G$ is $3${\it -connected} when there exists at least $3$ internally disjoint paths in $G$ between any two vertices in $G$. A {\it tree} is a connected graph with $n$ vertices and $n-1$ edges. $K_{n}$ denotes the {\it complete graph} on $n$ vertices. All basic graph theory definitions can be found in standard texts such as \cite{BM}, \cite{D} or \cite{GG}. All graphs in this work are assumed to be simple, undirected and finite.

\section{Constructing $3$-connected graphs}

Barnette and Gr\"unbaum introduce three operations on $3$-connected graphs in \cite{BG2} which are used to construct a $3$-connected graph from $K_4$ where all intermediate graphs are also $3$-connected. These operations are collectively referred to as {\it BG-operations} and are described in Definition \ref{BGdef}. Each operation is also individually named based on the number of additional vertices and edges, respectively, that are added to the $3$-connected graph by the operation i.e. the $(2,3)$-operation on a $3$-connected graph $H$ adds $2$ vertices to $V_{H}$ and $3$ edges to $E_{H}$. Before defining BG-operations note that an edge $uw\in E_{G}$ is {\it subdivided} whenever $uw$ is removed from $E_{G}$, a vertex $v$ is added to $V_{G}$ and the edges $uv$ and $vw$ are added to $E_{G}$.

\begin{defn}\label{BGdef} 
Given a $3$-connected graph $H$ then a BG-operation on $H$ is one of the following operations: 
\begin{itemize}
\item$(0,1)$-operation: add an edge $ab$ to $H$ such that $a,b\in V_H$ but $ab\not\in E_{H}$.
\item$(1,2)$-operation: subdivide an edge $xy\in E_H$ by adding a vertex $a\not\in V_H$, then adding the edge $ab\not\in E_H$ such that $b\in V_{H}$ and $b\neq x,y$.
\item$(2,3)$-operation: subdivide edges $xy$ and $wz$, where $xy\neq wz$, by adding vertices $a$ and $b$, respectively, then adding the edge $ab$.
\end{itemize}
\end{defn}

In the literature the operations described in Definition \ref{BGdef}, and shown in {\sc Figure} \ref{Cases2}, are called {\it basic} BG-operations as they do not result in the introduction of parallel edges.

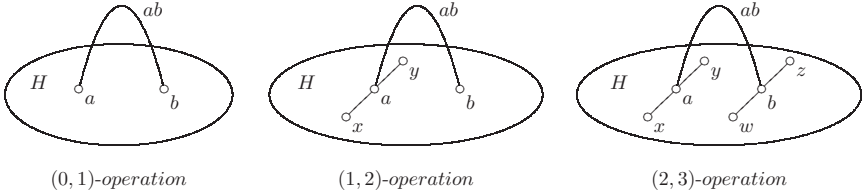
\begin{figure}[h]
\begin{center}
\scalebox{0.75}{$\begin{xy}\POS (0,5) *\cir<2pt>{} ="a" *+!UL{a},
(15,5) *\cir<2pt>{} ="b" *+!UL{b},
  (-7,5)*+!{H},
  (13,18)*+!{ab},
  (7,-12)*+!{(0,1){\text -operation}},

\POS "a" \ar@/^3.5pc/@{-}  "b",

\POS(7,4),  {\ellipse(20,9)<>{}},

\POS (57,10) *\cir<2pt>{} ="y" *+!UL{y},
(47,0) *\cir<2pt>{} ="x" *+!UL{x},
(52,5) *\cir<2pt>{} ="a" *+!UL{a},
(67,5) *\cir<2pt>{} ="b" *+!UL{b},
  (40,5)*+!{H},
  (65,18)*+!{ab},
   (57.5,-12)*+!{(1,2){\text -operation}},

\POS "y" \ar@{-}  "a",
\POS "a" \ar@{-}  "x",
\POS "a" \ar@/^3.5pc/@{-}  "b",

\POS(57.5,4),  {\ellipse(24,9)<>{}},

 \POS (110,10) *\cir<2pt>{} ="y" *+!UL{y},
(100,0) *\cir<2pt>{} ="x" *+!UL{x},
(105,5) *\cir<2pt>{} ="a" *+!UL{a},
(125,10) *\cir<2pt>{} ="z" *+!UL{z},
(115,0) *\cir<2pt>{} ="w" *+!UL{w},
(120,5) *\cir<2pt>{} ="b" *+!UL{b},
  (95,5)*+!{H},
  (118,18)*+!{ab},
   (112.5,-12)*+!{(2,3){\text -operation}},

\POS "y" \ar@{-}  "a",
\POS "a" \ar@{-}  "x",
\POS "z" \ar@{-}  "b",
\POS "b" \ar@{-}  "w",
\POS "a" \ar@/^3.5pc/@{-}  "b",

\POS(112.5,4),  {\ellipse(25,9)<>{}},

 \end{xy}$}
\caption{The BG-operations.}
\label{Cases2}
\end{center}
\end{figure}

\begin{thm}\label{BGthrm} A graph $G$ is $3$-connected if and only if it can be constructed from $K_{4}$ using BG-operations. \end{thm}

No loss of generality is incurred by using only basic BG-operations as every multigraph has a maximal {\it simple} graph which can be constructed from $K_4$ using {\it basic} BG-operations. Hence, it simply requires the addition of the appropriate parallel edges and loops to complete the construction of the required multigraph. Further details on Theorem \ref{BGthrm} can be found in \cite{BG2} and \cite{T2} as well as in \cite{Sch}.

\section{Listing all $3$-connected graphs}\label{s4}

Let $\mathcal{G}_3$ be the set of all (unlabelled) $3$-connected simple graphs. Define a graph equivalence relation $\sim$ as $G\sim H$ if and only if $|V_G|=|V_H|$ and $|E_G|=|E_H|$.\\

Consider a matrix $P=\big( p_{i,j}\big)_{i \in \mathbb{Z} ,j\in \mathbb{N}_0}$ such that each non-empty entry $p_{i,j}\in P$ contains a single equivalence class (with respect to $\sim$) with the entry $p_{0,0}$ containing $K_4$. \\

The equivalence classes are arranged in $P$ as follows:
\begin{itemize}
\item Each $G\in p_{i,j}$ is a maximal proper subgraph of some $G'\in p_{i,j+1}$ such that $V_{G}\subset V_{G'}$. 
\item Each $G\in p_{i,j}$ is a maximal proper subgraph of some $G''\in p_{i+1,j}$ such that $V_{G}= V_{G''}$ and $E_{G}\subset E_{G''}$. 
\item If $p_{i,j}$ contains a maximal complete graph, then $p_{i+1,j},\; p_{i+2,j}, \; ...$ are empty.
\end{itemize}

Such an arrangement of equivalence classes within $P$ defines the following relationships between vertex and edge set cardinalities: 
\begin{itemize}
\item Assuming that $p_{i,j}$ does not contain a maximal complete graph, then for any $G\in p_{i,j}$ there exists some $G^{*}\in p_{i+1,j}$ such that $G\subset G^{*}$ with $|V_{G^{*}}|= |V_{G}|$ and $ |E_{G^{*}}|=|E_{G}|+1$. See {\sc Figure} \ref{Poly}.
\item For any $G\in p_{i,j}$ there exists some $G^{**}\in p_{i,j+1}$ such that $G\subset G^{**}$ with $ |V_{G^{**}}|=|V_{G}|+1$ and $ |E_{G^{**}}|=|E_{G}|+3$ (when adding a vertex to a $3$-connected graph, so that $3$-connectedness is preserved, at least $3$ additional edges must also be added). See {\sc Figure} \ref{Poly}. 
\end{itemize}

 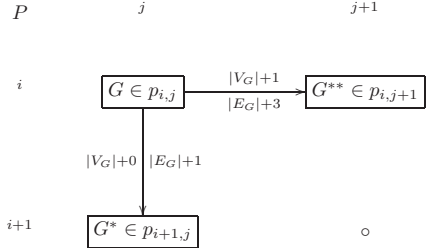
\begin{figure}[h]
\begin{center}\scalebox{0.75}{$
\xymatrix{P & ^{j} & & ^{j+1} &    \\
 ^{i}   &    *+[F]{G\in p_{i,j} }  \ar@{->}^{|V_G|+1}_{|E_G|+3}[rr] \ar@{->}_{|V_G|+0}^{|E_G|+1}[dd] & & *+[F]{G^{**}\in p_{i,j+1}}  &  \\
 & & & \\
^{i+1}   &     *+[F]{G^{*}\in p_{i+1,j} } & &  \circ   & \\ 
}$ }  
\caption{The relationship between graphs in adjacent entries in $P$.}
\label{Poly}
\end{center}
\end{figure}
 
For each $G\in p_{i,j}$ it is possible to state $|V_G|$ and $|E_G|$ in terms of $i$ and $j$.\\

As $p_{0,0}$ contains $K_4$ and all entries in the same column of $P$ contain graphs with the same number of vertices, then column $j$ contains graphs with $j+4$ vertices, where $j\in\mathbb{N}_0$. \\

Recall that $p_{0,0}$ contains $K_4$ and that $|E_{K_4}|={4\choose 2}$. So that $3$-connectedness is preserved then graphs in $p_{0,1}$ have $4+1$ vertices and ${4\choose 2}+3$ edges. It follows that graphs in entry $p_{0,j}$ have $4+j$ vertices and ${4\choose 2}+3j$ edges. Observe that graphs in $p_{1,j}$ (if they exist) contain $1$ more edge than graphs in $p_{0,j}$ and that graphs in $p_{-1,j}$ (if they exist) contain $1$ edge less than graphs in $p_{0,j}$. As rows are labelled using $\mathbb{Z}$ then the entry $p_{i,j}$ contains graphs with ${4\choose 2}+3j+i$ edges.   \\

It is now possible to define $P_{\mathcal{G}_{3}}$, the {\it partition matrix of} $\mathcal{G}_{3}$.

\begin{defn} Let $\mathcal{G}_{3}$ be the set of all $3$-connected simple graphs, then $P_{\mathcal{G}_{3}}$, the partition matrix of $\mathcal{G}_{3}$, is the matrix \[P_{\mathcal{G}_{3}} :=\big( p_{i,j}\big)_{i\in \mathbb{Z}, \; j\in \mathbb{N}_0}\]
such that \[ p_{i,j} :=\{G\in \mathcal{G}_{3} \; \mid \;  |V_{G}|=j+4\;\; {\textrm and } \;\; |E_{G}|= i +3j + 6  \}.\]
\end{defn}

A portion of $P_{\mathcal{G}_{3}}$ in shown in {\sc Figure} \ref{Gamma}.

\begin{figure}[h]
\[\scalebox{0.9}{$P_{\mathcal{G}_{3}}= \left( \begin{array}{ccccc}
\vdots & \vdots & \vdots &\vdots &  \\
p_{-4,0}& p_{-4,1} & p_{-4,2} & p_{-4,3} & \cdots \\
p_{-3,0}& p_{-3,1} & p_{-3,2} & p_{-3,3} & \cdots \\
p_{-2,0} & p_{-2,1} & p_{-2,2} & p_{-2,3} & \cdots \\
p_{-1,0} & p_{-1,1} & p_{-1,2} & p_{-1,3} & \cdots \\
p_{0,0} & p_{0,1} & p_{0,2} & p_{0,3} & \cdots \\
p_{1,0} & p_{1,1} & p_{1,2} & p_{1,3} & \cdots \\
p_{2,0} & p_{2,1} & p_{2,2} & p_{2,3} & \cdots \\
p_{3,0}& p_{3,1} & p_{3,2} & p_{3,3} & \cdots \\
p_{4,0}& p_{4,1} & p_{4,2} & p_{4,3} & \cdots \\
\vdots & \vdots & \vdots &\vdots & \ddots \\
\end{array} \right) = \left( \begin{array}{ccccc}
\vdots & \vdots & \vdots &\vdots &  \\
\varnothing & \varnothing & \varnothing & p_{-4,3} & \cdots \\
\varnothing & \varnothing & p_{-3,2} & p_{-3,3} & \cdots \\
\varnothing & \varnothing & p_{-2,2} & p_{-2,3} & \cdots \\
\varnothing & p_{-1,1} & p_{-1,2} & p_{-1,3} & \cdots \\
K_4 & p_{0,1} & p_{0,2} & p_{0,3} & \cdots \\
\varnothing & K_5 & p_{1,2} & p_{1,3} & \cdots \\
\varnothing & \varnothing & p_{2,2} & p_{2,3} & \cdots \\
\varnothing & \varnothing & K_6 & p_{3,3} & \cdots \\
\varnothing & \varnothing & \varnothing & p_{4,3} & \cdots \\
\vdots & \vdots & \vdots &\vdots & \ddots \\
\end{array} \right) $} \] 
\caption{The partition matrix $P_{\mathcal{G}_{3}}$ of $\mathcal{G}_{3}$.}
\label{Gamma}
\end{figure}

\section{BG-operations and $P_{\mathcal{G}_{3}}$}

It is now possible to see how entries in $P_{\mathcal{G}_{3}}$ are related in terms of the three BG-operations.
\begin{itemize}
\item Performing a $(0,1)$-operation on some $G\in p_{i,j}$ results in some $G'\in p_{i+1,j}$ such that $|V_{G'}|=|V_{G}|$ and $|E_{G'}|=|E_{G}|+1$. 
\item Performing a $(1,2)$-operation on some $G\in p_{i,j}$ results in some $G''\in p_{i-1,j+1}$ such that $|V_{G''}|=|V_{G}|+1$ and $|E_{G''}|=|E_{G}|+2$.
\item Performing a $(2,3)$-operation on some $G\in p_{i,j}$ results in some $G'''\in p_{i-3,j+2}$ such that $|V_{G'''}|=|V_{G}|+2$ and $|E_{G'''}|=|E_{G}|+3$. 
\end{itemize}
The relationship described in the three points above is illustrated in {\sc Figure} \ref{Poly2}.

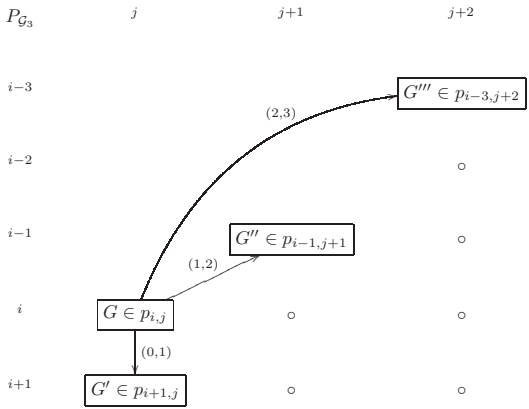
\begin{figure}[h]
\begin{center}\scalebox{0.7}{$
\xymatrix{P_{\mathcal{G}_{3}} & ^{j} & ^{j+1} & ^{j+2} &  \\
^{i-3}   & & & *+[F]{G'''\in p_{i-3,j+2}} &    \\
^{i-2}   &  &  & \circ &   \\
^{i-1}   &  & *+[F]{G'' \in p_{i-1,j+1}} & \circ &  \\
 ^{i}   &    *+[F]{G\in p_{i,j} }  \ar@/^3.2pc/@{->}^>>>>>>>>>>>>{(2,3)}[uuurr]\ar@{->}^{(1,2)}[ur] \ar@{->}^{(0,1)}[d] & \circ &  \circ &   \\
^{i+1}   &     *+[F]{G'\in p_{i+1,j} } &  \circ & \circ &   \\ 
}$ }  
\caption{The relationship between BG-operations and entries in $P_{\mathcal{G}_{3}}$.}
\label{Poly2}
\end{center}
\end{figure}

When constructing any $3$-connected graph $G$ using BG-operations it is worth noting that the order in which the BG-operations are performed is not always arbitrary as is illustrated by the following example. 

\begin{ex}
Consider the entry $p_{-2,2}$ in $P_{\mathcal{G}_{3}} $ as shown in {\sc Figure} \ref{Gamma}. The graph $G_1$ in {\sc Figure} \ref{g1g2} is the result of a $(1,2)$-operation performed on $K_4$ followed by another $(1,2)$-operation. The graph $G_2$ in {\sc Figure} \ref{g1g2} is the result of a $(2,3)$-operation performed on $K_4$ followed by a $(0,1)$-operation. Note that no graph in $p_{-2,2}$ is the result of a $(0,1)$-operation performed on $K_4$ followed by a $(2,3)$-operation. Observe that $G_1\sim G_2$ as they are both contained in $p_{-2,2}$ but that it is not possible to construct $G_{1}$ using a $(2,3)$-operation followed by a $(0,1)$-operation and similarly, it is not possible to construct $G_{2}$ using two successive $(1,2)$-operations.

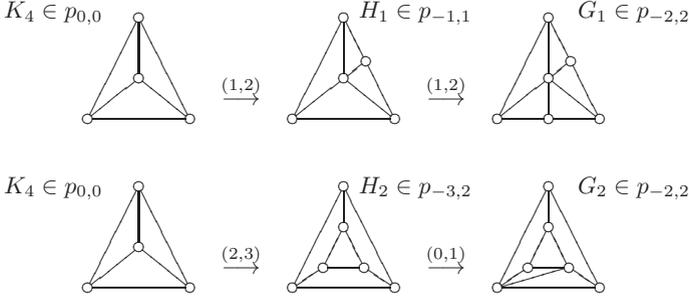
\begin{figure}[h]
\begin{center}
\scalebox{0.9}{$\begin{xy}
 \POS (0,10) *\cir<2pt>{} ="a" *+!D{} ,
 (15,10) *\cir<2pt>{} ="b" *+!R{},
 (7.5,25) *\cir<2pt>{} ="c" *+!R{},
 (7.5,16) *\cir<2pt>{} ="d" *+!L{},
 (-5,25) *+!{K_4\in p_{0,0}},
 (22.5,12) *+!{\stackrel{(1,2)}{\longrightarrow}}
 
\POS "a" \ar@{-}  "b",
\POS "a" \ar@{-}  "c",
\POS "a" \ar@{-}  "d",
\POS "b" \ar@{-}  "c",
\POS "b" \ar@{-}  "d",
\POS "c" \ar@{-}  "d",

 \POS (30,10) *\cir<2pt>{} ="a" *+!D{} ,
 (45,10) *\cir<2pt>{} ="b" *+!R{},
 (37.5,25) *\cir<2pt>{} ="c" *+!R{},
 (37.5,16) *\cir<2pt>{} ="d" *+!L{},
  (40.75,18.5) *\cir<2pt>{} ="e" *+!L{},
   (48,25) *+!{H_1 \in p_{-1,1}},
  (52.5,12) *+!{\stackrel{(1,2)}{\longrightarrow}}
 
\POS "a" \ar@{-}  "b",
\POS "a" \ar@{-}  "c",
\POS "a" \ar@{-}  "d",
\POS "b" \ar@{-}  "e",
\POS "e" \ar@{-}  "c",
\POS "b" \ar@{-}  "d",
\POS "c" \ar@{-}  "d",
\POS "e" \ar@{-}  "d",

 \POS (60,10) *\cir<2pt>{} ="a" *+!D{} ,
 (75,10) *\cir<2pt>{} ="b" *+!R{},
 (67.5,25) *\cir<2pt>{} ="c" *+!R{},
 (67.5,16) *\cir<2pt>{} ="d" *+!L{},
  (70.75,18.5) *\cir<2pt>{} ="e" *+!L{},
   (67.5,10) *\cir<2pt>{} ="f" *+!L{},
 (80,25) *+!{G_1 \in p_{-2,2}},
 
\POS "a" \ar@{-}  "f",
\POS "b" \ar@{-}  "f",
\POS "d" \ar@{-}  "f",
\POS "a" \ar@{-}  "c",
\POS "a" \ar@{-}  "d",
\POS "b" \ar@{-}  "e",
\POS "e" \ar@{-}  "c",
\POS "b" \ar@{-}  "d",
\POS "c" \ar@{-}  "d",
\POS "e" \ar@{-}  "d",

 \POS (0,-15) *\cir<2pt>{} ="a" *+!D{} ,
 (15,-15) *\cir<2pt>{} ="b" *+!R{},
 (7.5,0) *\cir<2pt>{} ="c" *+!R{},
 (7.5,-9) *\cir<2pt>{} ="d" *+!L{},
 (-5,-1) *+!{K_4 \in p_{0,0}},
  (22.5,-13) *+!{\stackrel{(2,3)}{\longrightarrow}}
 
\POS "a" \ar@{-}  "b",
\POS "a" \ar@{-}  "c",
\POS "a" \ar@{-}  "d",
\POS "b" \ar@{-}  "c",
\POS "b" \ar@{-}  "d",
\POS "c" \ar@{-}  "d",

 \POS (30,-15) *\cir<2pt>{} ="a" *+!D{} ,
 (45,-15) *\cir<2pt>{} ="b" *+!R{},
 (37.5,0) *\cir<2pt>{} ="c" *+!R{},
  (37.5,-6) *\cir<2pt>{} ="e" *+!L{},
   (34.5,-12) *\cir<2pt>{} ="f" *+!L{},
    (40.5,-12) *\cir<2pt>{} ="g" *+!L{},
     (48,-1) *+!{H_2 \in p_{-3,2}},
  (52.5,-13) *+!{\stackrel{(0,1)}{\longrightarrow}}
 
\POS "a" \ar@{-}  "b",
\POS "a" \ar@{-}  "c",
\POS "a" \ar@{-}  "f",
\POS "b" \ar@{-}  "c",
\POS "b" \ar@{-}  "g",
\POS "c" \ar@{-}  "e",
\POS "f" \ar@{-}  "e",
\POS "g" \ar@{-}  "e",
\POS "f" \ar@{-}  "g",

 \POS (60,-15) *\cir<2pt>{} ="a" *+!D{} ,
 (75,-15) *\cir<2pt>{} ="b" *+!R{},
 (67.5,0) *\cir<2pt>{} ="c" *+!R{},
  (67.5,-6) *\cir<2pt>{} ="e" *+!L{},
   (64.5,-12) *\cir<2pt>{} ="f" *+!L{},
    (70.5,-12) *\cir<2pt>{} ="g" *+!L{},
 (80,-1) *+!{G_2 \in p_{-2,2}},
 
\POS "a" \ar@{-}  "b",
\POS "a" \ar@{-}  "c",
\POS "a" \ar@{-}  "f",
\POS "b" \ar@{-}  "c",
\POS "b" \ar@{-}  "g",
\POS "c" \ar@{-}  "e",
\POS "f" \ar@{-}  "e",
\POS "g" \ar@{-}  "e",
\POS "f" \ar@{-}  "g",
\POS "a" \ar@{-}  "g",

 \end{xy}$ }
\caption{The choice of BG-operations (type or order) when constructing a $3$-connected graph is not arbitrary.}
\label{g1g2}
\end{center}
\end{figure}

  \end{ex}

\section{Degree sequences and graph partitions}\label{s5}

A finite sequence $s=\{s_1,...,s_n\}$ of non-negative integers is called {\it graphic} if there exists a finite simple graph with vertex set  $\{v_1,..., v_n\}$ such that $v_i$ has degree $s_i$ for all $i=1,...,n$. Necessary and sufficient conditions for a sequence of non-negative integers to be graphic were first described by Erd\"os and Gallai in \cite{EG}. Necessary and sufficient conditions for a sequence of non-negative integers to be connected are implicit in Hakimi \cite{Hk} and these conditions have been stated explicitly by the author in \cite{Me15}. The maximum degree of a vertex in $G$ is denoted $\Delta_G$ and the minimum degree of a vertex in $G$ is denoted $\delta_G$. Given a graph $G$ then the degree sequence $d(G)$ is the monotonic non-increasing sequence of degrees of the vertices in $V_G$. This means that every graphical sequence $s$ is equal to the degree sequence $d(G)$ of some graph $G$ (subject to possible rearrangement of the terms in $s$).

\begin{defn} A finite sequence $s=\{s_1,...,s_n\}$ of positive integers is called {\it $3$-connected} if there exists a finite simple $3$-connected graph with vertex set  $\{v_1,..., v_n\}$ such that $deg(v_i)= s_i$ for all $i=1,...,n$.
\end{defn}

Given a sequence of positive integers $s=\{s_1,...,s_n\}$ then define the {\it associated pair of $s$}, denoted $(\varphi(s),\epsilon(s))$, to be the pair $(n, \frac{1}{2}\sum\limits_{i=1}^{n}s_i)$. Where no ambiguity can arise,  $(\varphi(s),\epsilon(s))$ is simply denoted $(\varphi,\epsilon)$.

\begin{lem}\label{Ent} Given a sequence $s$ with associated pair $(\varphi,\epsilon)$ such that $s=d(G)$ for some $3$-connected graph $G\in \mathcal{G}_{3}$, then $G$ is contained in $p_{\;\epsilon-3\varphi +6, \; \varphi -4}$ in $P_{\mathcal{G}_{3}}$. 
\end{lem}

\begin{proof} As $s$ is the degree sequence $d(G)$ of some graph $G\in \mathcal{G}_{3}$ then $|V_G|$ is the number of terms in $s$ which is $\varphi$ and $|E_G|$ is half the sum of the degrees of all vertices in $G$ which is exactly $\epsilon$. Recall that each entry $p_{i,j}$ in $P_{\mathcal{G}_{3}}$ is defined as $\{G\in \mathcal{G}_{3} \; \mid \;  |V_{G}|=j+4\;\; {\textrm and } \;\; |E_{G}|= i+3j+6  \} $. By rearranging $|V_{G}|=\varphi=j+4$ and substituting into $|E_{G}|=\epsilon= i+3j+6 $ then $j=\varphi-4$ and $i=\epsilon -3\varphi +6$ hence $s=d(G)$ for some $G$ contained in $p_{\;\epsilon -3\varphi +6,\;\varphi-4}$.  
\end{proof}

\begin{lem}\label{EmpEnt} The non-empty entries in column $j=\varphi -4$ of $P_{\mathcal{G}_{3}}$ are 
\begin{itemize}
\item $p_{\frac{-3\varphi +12}{2},\; \varphi-4}$ to $p_{{\varphi -3 \choose 2},\; \varphi -4}$, inclusive, when $\varphi$ is even, and  
\item $p_{\frac{-3\varphi +13}{2},\; \varphi-4}$ to $p_{{\varphi -3 \choose 2},\; \varphi -4}$, inclusive, when $\varphi$ is odd.
\end{itemize}
\end{lem}
\begin{proof} 
Observe that $\delta_G\geq 3$ for every $G\in \mathcal{G}_{3}$. It follows that the minimum possible $\epsilon$ for a degree sequence $s$ of any $3$-connected graph is $\epsilon=\frac{3\varphi}{2}$ i.e. $s=\underbrace{\{3,...,3\}}_{\varphi}$. Note that $\epsilon$ must be even by definition, however, this can only occur whenever $\varphi$ is even. In other words, there cannot exist a $3$-regular graph with an odd number of vertices. From Lemma \ref{Ent}, when $\varphi$ is even, the uppermost non-empty entry in column $j=\varphi -4$ of $P_{\mathcal{G}_{3}}$ is contained in row $\min\{\epsilon\} -3\varphi +6=\frac{3\varphi}{2}-3\varphi +6=\frac{-3\varphi+12}{2}$. \\

Observe that for any $G\in \mathcal{G}_{3}$ where $d(G)=s$ then $\Delta_G\leq \varphi-1$ as $G$ is simple. It follows that the maximum possible $\epsilon$ of any $3$-connected graph is $\epsilon={\varphi \choose 2}$ i.e. $s=\underbrace{\{\varphi-1,...,\varphi-1\}}_{\varphi}$. From Lemma \ref{Ent} the lowermost non-empty entry in column $j=\varphi -4$ of $P_{\mathcal{G}_{3}}$ is contained in row $\max\{\epsilon\} -3\varphi +6={\varphi \choose 2}-3\varphi +6=\frac{\varphi^2-7\varphi +12}{2}={\varphi-3 \choose 2}$. This argument is summarised in {\sc Figure} \ref{table1}.\\

 \begin{figure}[h]
{\renewcommand{\arraystretch}{1.5} 
\[ \begin{array}{c|l|c} 
 \varphi \; \textrm{even}  & \{s_1,\dots, s_n\} &   \epsilon \\ \hline
p_{\frac{-3\varphi +12}{2},\; \varphi-4} & \{3,\dots,3\}  &   \frac{3\varphi}{2}  \\
\vdots & \hspace{0.5cm}\vdots  &    \vdots \\
p_{\;{\varphi-3 \choose 2},\; \varphi -4} & \{n-1,\dots,n-1\} &  {\varphi \choose 2}  \\
\end{array}  \] }

\caption{All possible $3$-connected degree sequences of even length $\varphi$.}
\label{table1}
\end{figure}

If $\varphi$ is odd, note that the minimum possible $\epsilon$ for a degree sequence $s$ of any $3$-connected graphs is $\epsilon=\frac{3\varphi+1}{2}$ i.e. $s=\underbrace{\{4,3,...,3\}}_{\varphi}$. It follows from Lemma \ref{Ent} that, when $\varphi$ is odd, the uppermost non-empty entry in column $j=\varphi -4$ of $P_{\mathcal{G}_{3}}$ is contained in row $\min\{\epsilon\} -3\varphi +6=\frac{3\varphi+1}{2}-3\varphi +6=\frac{-3\varphi+13}{2}$. As the parity of $\varphi$ is irrelevant when maximising $\epsilon$ then the lowermost non-empty entry in column $j=\varphi -4$ of $P_{\mathcal{G}_{3}}$ is contained in row ${\varphi-3 \choose 2}$. This argument is summarised in {\sc Figure} \ref{table2}.

\begin{figure}[h]
{\renewcommand{\arraystretch}{1.5} 
\[  \begin{array}{c|l|c}
 \varphi \; \textrm{odd} & \{s_1,\dots, s_n\} &   \epsilon \\ \hline
p_{\frac{-3\varphi +13}{2},\; \varphi-4} & \{4,3,\dots,3\}  &  \frac{3\varphi +1}{2} \\
\vdots & \hspace{0.5cm}\vdots  &    \vdots \\
p_{\;{\varphi-3 \choose 2},\; \varphi -4} & \{n-1,\dots,n-1\} &   {\varphi \choose 2}  \\
\end{array}  \] }

\caption{All possible $3$-connected degree sequences of odd length $\varphi$.}
\label{table2}
\end{figure}

\end{proof}

\begin{cor}\label{EntCorr} The non-empty entries in column $j$ of $P_{\mathcal{G}_{3}}$ are 
\begin{itemize}
\item $p_{\frac{-3j}{2},\; j}$ to $p_{\frac{j^2 + j}{2},\; j}$, inclusive, when $j$ is even, and  
\item $p_{\frac{-3j+1}{2},\; j}$ to $p_{\frac{j^2 + j}{2},\; j}$, inclusive, when $j$ is odd.
\end{itemize}
\end{cor}
\begin{proof} The result follows from Lemma \ref{EmpEnt} by letting $\varphi=j+4$.
\end{proof}

It follows that, when $j$ is odd, column $j$ in $P_{\mathcal{G}}$ has $\frac{j^2 - j}{2}+\frac{3j - 1}{2}+1=\frac{j^2+4j+1}{2}$ non-empty entries, for example, column $j=1$ has $\frac{1^2 +4(1)+1}{2}=3$ non-empty entries, as shown in {\sc Figure} \ref{Gamma}. Similarly, when $j$ is even, column $j$ in $P_{\mathcal{G}}$ has $\frac{j^2 - j}{2}+\frac{3j }{2}+1=\frac{j^2+4j+2}{2}$ non-empty entries, for example, column $j=2$ has $\frac{2^2 +4(2)+2}{2}=7$ non-empty entries, again see {\sc Figure} \ref{Gamma}.\\

Note that it is possible to provide alternative justifications for Lemma \ref{EmpEnt} and Corollary \ref{EntCorr} using the relationship between BG-operations and entries in $P_{\mathcal{G}_3}$ which is outlined in {\sc Figure} \ref{Poly2}.

\section{Results}

\begin{thm}\label{Main} Given a sequence $s=\{s_1,...,s_n\}$ of positive integers, with the associated pair $(\varphi, \epsilon)$, such that $s_i\geq s_{i+1}$ for $i=1,...,n-1$ then $s$ is $3$-connected if and only if 
\begin{itemize}
\item $\epsilon \in \mathbb{N}$,
\item $\frac{3\varphi}{2} \leq \epsilon \leq {{\varphi \choose 2}}$, 
\item $s_1\leq \varphi-1$ and $s_n\geq 3$.
\end{itemize}
\end{thm}

\begin{proof} ($\Rightarrow$) Clearly $\epsilon \in \mathbb{N}$ is a necessary condition for any sequence $s$ to be realisable as half the sum of the degrees in any graph is the number of edges in that graph and this must be a natural number.  The necessity of the condition $ \frac{3\varphi}{2}  \leq \epsilon \leq {{\varphi \choose 2}}$ follows from Lemma \ref{EmpEnt} noting that $\epsilon > \frac{3\varphi}{2}$ whenever $\varphi$ is odd. Finally, the necessity of the condition $s_1\leq \varphi-1$ follows from the definition of a simple graph and the need for $s_n\geq 3$ is due to the fact that every vertex in a $3$-connected graph has degree at least $3$. It follows from {\sc Figures} \ref{table1} and \ref{table2} in Lemma \ref{EmpEnt} that, for a fixed $\varphi$, all sequences with $\frac{3\varphi}{2} < \epsilon < {{\varphi \choose 2}}$ are also realisable.  \\
 
($\Leftarrow$) Suppose that $s=\{s_1,...,s_n\}$ is $3$-connected. This means that $s$ is the degree sequence of a $3$-connected graph $G$, hence $\sum\limits_{i=1}^{n}deg(v_i)=2|V_G|$ and so $\epsilon\in\mathbb{N}$. As $G$ is $3$-connected then $deg(v_i)\geq 3$ for all $i=1,...,n$ hence if $G$ is a minimal $3$-connected graph on $n$ vertices then $d(G)=\{3,...,3\}$ with $|E_{G}|=\frac{3n}{2}$ if $n$ is even or $d(G)=\{4,3,...,3\}$ with $|E_{G}|=\frac{3n+1}{2}$ if $n$ is odd, hence $s_n\geq 3$ and $\epsilon \geq \frac{3\varphi}{2}$. As $G$ is simple then $deg(v_i)\leq n-1$ for all $i=1,...,n$ and the maximal simple ($3$-connected) graph on $n$ vertices is the complete graph $K_n$ which has the degree sequence $\{n-1,...,n-1\}$ and $|E_{K_n}|={n \choose 2}$, hence $s_1\leq n-1$ and $\epsilon \leq {\varphi \choose 2}$.   
\end{proof}

Before stating the next result the following definition is required. 

\begin{defn} A finite sequence $s=\{s_1,...,s_n\}$ of positive integers is called {\it necessarily $3$-connected} if $s$ can only be realisable as a $3$-connected (simple) graph.
\end{defn}

\begin{thm}\label{Crry1} Given a sequence $s=\{s_1,...,s_n\}$ of positive integers, with the associated pair $(\varphi, \epsilon)$, such that $s_i\geq s_{i+1}$ for $i=1,...,n-1$ then $s$ is necessarily $3$-connected if and only if $s$ is $3$-connected and 
$\epsilon > {{\varphi -2 \choose 2}} + 5.$
\end{thm}

\begin{proof}  ($\Rightarrow$) Clearly it is necessary for $s$ to be $3$-connected if it is to be necessarily $3$-connected. It is required to show that it is necessary for $ \epsilon > {{\varphi -2 \choose 2}} +5$. Consider a sequence $s=\{s_1,...,s_n\}$ such that $ \epsilon = {{\varphi -2 \choose 2}} +5.$ Observe that one such sequence is $s'=\{n-1,n-1,n-3,...,n-3,3,3\}$ which has $(\varphi(s'),\epsilon(s'))=\left(n, \frac{(n-2)(n-3)+4+6}{2}\right)=\left(n,{n -2 \choose 2} +5\right)$. Observe that $s'=d(G_1)$, see {\sc Figure} \ref{graphs}, where $G_1= H_1\cup H_2$ such that $H_1\simeq K_4$, $H_2\simeq K_{n-2}$ and $H_1\cap H_2 \simeq K_2$ with $V_{H_1}=\{v_1,v_2, v_{n-1},v_n\}$, $V_{H_2}=\{v_1,...,v_{n-2}\}$ and $V_{H_1\cap H_2}=\{v_1,v_2\}$. Note that $G_1$ is $2$-connected as $G_1\setminus \{v_1,v_2\}$ is disconnected. \\

However, $s'=\{n-1, n-1, n-3,...,n-3,3,3\}$ is, in fact, $3$-connected as $s'$ is also the degree sequence of $G_2$, again see {\sc Figure} \ref{graphs}, noting that $v_iv_j\in E_{G_1}$ but $v_iv_j\not\in E_{G_2}$. Therefore, it is required that $ \epsilon > {{\varphi -2 \choose 2}} +5$ if $s$ is to be necessarily $3$-connected as the sequence $s'=\{n-1,n-1,n-3,...,n-3,3,3\}$ with $\epsilon(s')={{\varphi -2 \choose 2}} +5$ is realisable as a $2$-connected graph.  \\

\begin{figure}[h]
\begin{center}
\scalebox{0.9}{$\begin{xy}\POS (1,0) *\cir<2pt>{} ="a" *+!L{\;v_n},
 (-7.5,0) *\cir<2pt>{} ="b" *+!R{v_{n-1}},
(10.5,8) *\cir<2pt>{} ="c" *+!DR{v_1},
(10.5,-8) *\cir<2pt>{} ="d" *+!UR{v_2},
 (29.5,8) *\cir<2pt>{} ="e" *+!DL{v_i},
(29.5,-8) *\cir<2pt>{} ="f" *+!UL{v_j},
  (-5,12)*+!{G_1},
  (19,0)*+!{K_{n-2}},
  
\POS "a" \ar@{-}  "b",
\POS "a" \ar@{-}  "c",
\POS "a" \ar@{-}  "d",
\POS "b" \ar@{-}  "c",
\POS "b" \ar@{-}  "d",
\POS "c" \ar@{-}  "d",
\POS "e" \ar@{-}  "f",

\POS(20,0),  {\ellipse(12.5,12.5)<>{}},

\POS (71,0) *\cir<2pt>{} ="a" *+!L{\;v_n},
 (62.5,0) *\cir<2pt>{} ="b" *+!R{v_{n-1}},
(80.5,8) *\cir<2pt>{} ="c" *+!DR{v_1},
(80.5,-8) *\cir<2pt>{} ="d" *+!UR{v_2},
 (99.5,8) *\cir<2pt>{} ="e" *+!DL{v_i},
(99.5,-8) *\cir<2pt>{} ="f" *+!UL{v_j},
  (65,12)*+!{G_2},
  
\POS "a" \ar@{-}  "e",
\POS "b" \ar@{-}  "f",
\POS "a" \ar@{-}  "c",
\POS "a" \ar@{-}  "d",
\POS "b" \ar@{-}  "c",
\POS "b" \ar@{-}  "d",
\POS "c" \ar@{-}  "d",

\POS(90,0),  {\ellipse(12.5,12.5)<>{}},

 \end{xy}$}

\caption{ $d(G_1)=d(G_2)=s'=\{n-1,n-1,n-3,...,n-3,3,3\}$.}
\label{graphs}
\end{center}
\end{figure}
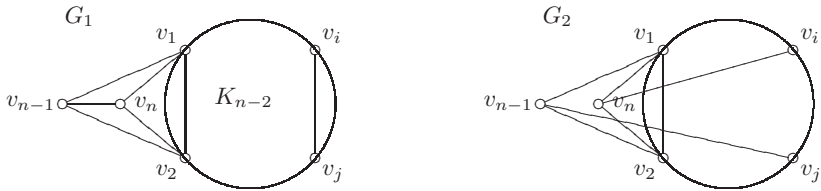

($\Leftarrow$) It is now required to show that if $s$ is $3$-connected and $ \epsilon > {{\varphi -2 \choose 2}} +5$ then $s$ is necessarily $3$-connected. To show this it is required to show that the maximum number of edges in a graph with $n$ vertices which is not $3$-connected is ${n-2 \choose 2} +5$. The graph $G_1$ in {\sc Figure} \ref{graphs} shows that such a graph exists, and so it remains to show that a graph with $ \epsilon = {{\varphi -2 \choose 2}} +5$ is maximally $2$-connected i.e. adding one edge will always result in a $3$-connected graph. \\
 
Observe that any maximally $2$-connected graph on $n$ vertices will necessarily contain a cut set with two vertices $\{u,v\}$ i.e. $G\setminus \{u,v\}$ is disconnected.  To maximise the number of edges in $G$ it is clear that $G\setminus \{u,v\}$ contains just two connected components i.e. $G=H_1 \cup H_2$ where $H_1\cap H_2=(\{u,v\},\{uv\})=C$ with $H_1\simeq K_{a+2}$, $H_2\simeq K_{b+2}$ and $H_1\cap H_2 \simeq K_2$ (noting that $a+b=n-2$). So, the task of maximising $|E_{H_1\setminus C}|+|E_{H_2\setminus C}|$ is equivalent to minimising the number of edges in a complete bipartite graph $K_{a,b}$ as $K_{n}\setminus (E_{H_1\cup H_2}) \simeq K_{a,b}$.\\

Let $a+b=n-2$, with $a\leq b$, then $|E_{K_{a,b}}|=ab$ where $a,b\in\{1,...,n-3\}$. Note that $a>0$ as $G\setminus \{u,v\}$ is disconnected i.e. $K_a\neq K_0=(\varnothing, \varnothing)$. It is straightforward to show that $ab$ attains its maximum at $a=b=\frac{n-2}{2}$, when $n$ is even, and at $a=\lfloor\frac{n-2}{2}\rfloor, b=\lceil\frac{n-2}{2}\rceil$ when $n$ is odd. It follows that $ab$ is minimised when $a=1$ and $b=n-3$. However, observe that $a>1$ as $a=1$ implies that $H_1\simeq K_3$ which means that $d(G)$ contains a term equal to $2$, but this contradicts the $s_n\geq 3$ condition. Hence $|E_{K_{a,b}}|$, with $a+b=n-2$, is minimised when $a=2$ and $b=n-4$ and so the maximal $2$-connected graph on $n$ vertices is isomorphic to $H_1\cup H_2$ where $H_1\simeq K_{4}, H_2\simeq K_{n-2}$ and $H_1\cap H_2\simeq K_2$, see $G_1$ in {\sc Figure} \ref{graphs}. Notice that the union of $G_1$ and any edge in $\overline{G_1}$, the complement of $G_1$, results in a $3$-connected graph. 
\end{proof}

\begin{cor} All simple graphs with $n$ vertices and at least $\frac{n^2-5n +18}{2}$ edges are $3$-connected.
\end{cor}
\begin{proof}  As shown in Theorem \ref{Crry1}, a maximal $2$-connected graph with $n$ vertices is isomorphic to the union of $H_1\simeq K_4$ and $H_2\simeq K_{n-2}$ where $H_1\cap H_2 \simeq K_2$ and such a graph has $|E_{K_{n-2}}|+|E_{K_4}|-|E_{K_2}| = {n -2 \choose 2} +{4 \choose 2} -{2 \choose 2}$ edges. It follows that any simple graph with $n$ vertices and at least $\left({{n -2 \choose 2}} +5\right)+1 = \frac{(n-2)(n-3)+12}{2}=\frac{n^2-5n+18}{2}$ edges is $3$-connected. 
\end{proof} 

Note that for all $n\in \mathbb{N}$, $n^2-5n$ is even and $n^2-5n+18>0$, hence $\frac{n^2-5n+18}{2}\in \mathbb{N}$.

\end{document}